\documentclass[11pt, reqno]{amsart}

\usepackage[utf8]{inputenc}

\usepackage{fontenc}

\usepackage{amsmath, amsfonts, amssymb, amsthm, framed}

\usepackage[margin=1in]{geometry}

\usepackage{xcolor, float}

\usepackage[colorlinks=true,allcolors=blue]{hyperref}

\usepackage[noabbrev,capitalize]{cleveref}

\crefname{equation}{}{}

\crefname{enumi}{Property}{Properties}

\usepackage{comment}

\usepackage{dsfont}

\usepackage{etoolbox}

\patchcmd{\section}{\scshape}{\scshape\bfseries}{}{}

\makeatletter

\renewcommand{\@secnumfont}{\scshape\bfseries}

\makeatother

\newtheorem{theorem}{Theorem}[section]

\newtheorem{lemma}[theorem]{Lemma}

\newtheorem{corollary}[theorem]{Corollary}

\newcommand{\R}{\mathbb{R}}

\newcommand{\inv}{^{-1}}

\newcommand{\pr}{\mathbb{P}}

\usepackage{color}

\newcommand{\ignore}[1]{}

\title{Anti-concentration with respect to random permutations}

\author{Aaron Berger, Ross Berkowitz, Pat Devlin, Van Vu}

\begin{document}

\maketitle

\begin{abstract}
Classical anti-concentration results focus on the  random sum $S := \sum _{i=1}^n \xi_i v_i$, where $\xi_i$ are independent random variables and $v_i$ are real numbers.  In this paper, we prove new concentration results concerning  the random sum $S := \sum_{i=1}^n w_{\pi _i  } v_i $, where $w_i , v_i$ are real numbers and $\pi$ is a random permutation.
\end{abstract}

\section {Introduction}
Anti-concentration is a well known and important phenomenon in probability. An anti-concentration inequality asserts that (under certain conditions) a random variable $S$ does not have too much mass in a small region or on a single point. We follow Levy  \cite{Levy} and define the concentration function of a real random variable $Y$ as
\begin{equation}\label{CF}
Q (S, t)= \sup_{|I|=t}  \pr (S \in I) , 
\end{equation} 
where $I$ runs over the set of all closed intervals of length $t$.

Many prominent  mathematicians, including Littlewood, Offord, Erd\H os, Levy, Kolmogorov, Rogozin,  and Esseen made significant contribution to the early study of  the concentration function and the anti-concentration phenomenon \cite{LO, Esseen, Rog, Kol1, Kol2, Levy, Erdos}. Since the early 2000s, the study of anti-concentration has been revitalized,  with motivations and applications coming from various areas, including  random matrix theory, combinatorics, random functions,  and data science. New, stronger, anti-concentration inequalities have been discovered, and these play essential role in the solution of many long standing problems, such as the Circular Law conjecture in random matrix theory; see \cite{NVsurvey} for a survey.

In most of the above mentioned works, the random variable in question is a linear combination of many independent atom random variables. The Erd\H{o}s-Littlewood-Offord inequality, one of the earliest and most well-known results in this area, has the following form.

\begin{theorem}[Erd\H{o}s-Littlewood-Offord] \label{LOE} Let $v \in \R^n$ satisfy $|v_i| \ge 1$ for all $i \in [n]$ and let $\xi_i$ be iid Rademacher variables. Then for any interval $I$
\begin{equation*}
\pr ( \sum_{i=1}^n \xi_i v_i  \in I) = O\left( \frac{| I| +1 } {\sqrt n }\right).
\end{equation*}             
\end{theorem}

Notice that the random variable $\sum_{i=1}^n \xi_i v_i $ can be written as the inner product of the random vector $w := (\xi_1, \dots, \xi_n)$ with the deterministic vector $v =(v_1, \dots, v_n)$.

In this, and many other theorems (see \cite{NVsurvey}), $w$ is a random vector distributed according to the product measure generated by independent atom random variables.  In applications, another frequently used probability space is the space generated by random permutations.  The goal of this paper is to initiate the study  of anti-concentration results in the random permutation space.

Let $w= (w_1, \dots, w_n) $ be a fixed vector and $\pi$ a permutation in $S_n$.  We denote by $w_{\pi} $ the vector  $(w_{\pi (1) }, \dots, w_{\pi (n) } ) $  obtained by permuting the coordinates of $w$  by $\pi$. Fix $v \in \R^n $. Our object of study is the random variable  $w_{\pi }  \cdot v = \sum_i  w_i v_{\pi(i)}$, where $\pi$ is chosen uniformly at random from $S_n$.

  In a recent paper \cite{Soz}, Soze, motivated by an  application  concerning random polynomials,  considered a special case wheren $v= (1,2, \dots, n)$,  and proved

 \begin{theorem}  \label{thm:soze_main}

                Let $v= (1,2,\dots, n) $ and $w \in \R^n$ be a unit vector such that $w \cdot {\bf 1 } =0$. Then for any $L \in \R$

                $$\pr \left( |w_\pi \cdot v - Ln|  \le 1\right) =        O\left( \frac{ 1} {n} e^{-c |L| }\right), $$ for some absolute constant $c > 0$.            \end{theorem}

 To give some intuition for the assumption that  $w \cdot {\bf 1 } =0$,  let us point out that if all the coordinates of $w$ are the same, then $w_\pi \cdot v $ is constant and no nontrivial anti-concentration statement holds. Thus, in statements involving random permutations, the anti-concentration of $w_\pi\cdot v$ will depend on how close $w$ and $v$ are to being parallel to $\bf 1$.

 As an analogue of Theorem \ref{LOE}, one has the following corollary.

 \begin{corollary}  \label{cor:soze_perm}

                Let $v= (1,2,\dots, n) $ and $w \in \R^n$ be a unit vector such that $w \cdot {\bf 1 } =0$. Then for any point $x \in \mathbb{R}$
                \begin{equation*}
               \pr ( w_\pi \cdot v =x ) =                 O\left( \frac{1 } {n} \right) .
               \end{equation*}


 \end{corollary}
\noindent The bound $O(1/n)$ here is tight, as shown by  $w= (1,-1, 0, \dots, 0)$.  The problem of bounding this point mass was recently considered by Huang, McKinnon, and Satriano \cite{HMS}, who obtained the correct asymptotics and optimal bounds in the case $n = p$ or $n = p+1$, where $p$ is a prime.  Subsequently, Pawlowski \cite{Paw} obtained optimal bounds for all $n$, showing that under the assumptions if the coordinates of $v$ are all distinct and $w \cdot {\bf 1} \neq 0$, then for all $x\in \mathbb{R}$ we have

\begin{equation*}
    \pr[w_\pi \cdot v = x] \le \begin{cases}\frac 1n  \qquad &\text{ if $n$ is odd,}\\
    \frac {1}{n-1} \qquad &\text{ if $n$ is even.}
    \end{cases}
\end{equation*}
In both cases there exist choices of $v,w$ for which equality holds.

 We are going to develop a bound that applies for arbitrary vectors $v$ and $w$.  Our approach will be entirely different from that of Soze, which is tailored to the special case $v= (1,2, \dots, n)$. We say that  $w$ is {\it increasing } if $w_1. \le w_2 \le \dots \le w_n$. Without loss of generality, we can assume that both $w$ and $v$ are increasing and define

 \begin{equation*} \label{delta} \Delta (v)=  \min_{i \neq j} | v_i -v_j | . \end{equation*}
Our main result is the following
\begin{theorem}[Anti-concentration for random permutations]\label{thm:anti_conc_perm}

                Let $w, v \in \R^n$ be increasing with $\Delta (v) >0$.  For any interval $I$ and indices $i_1, i_2$ such that  $i_1+i_2 \le n$ and  $w_{n-i_2}-w_{i_1} > 0$, we have

                $$\pr  (w_\pi \cdot v \in I)= O \left(1 + \frac{|I|}  {\Delta(v) (w_{n-i_2}-w_{i_1})}\right) \frac{1}{(i_1+i_2)\sqrt{\min(i_1,i_2)}}. $$

\end{theorem}

This gives a strong bound if by some careful choice we can make $i_1$, $i_2$, and $w_{n-i_2}-w_{i_1}$ large simultaneously.  In the case when there are repetition among the $v_i$, or $\Delta (v)$ is small, one can still get  a bound by conditioning on a subset which has large $\Delta$.  For instance, (by choosing $i_i = i_2 = \varepsilon n /3$ in what follows) we obtain the following improvement of Corollary \ref{cor:soze_perm}

\begin{corollary}
Assume $v$ has all distinct coordinates and that no $w_i$ is repeated more than $(1-\epsilon)n$ times, for some  $\epsilon >0$.  Then for any $x \in \mathbb{R},$

$$\pr (w_{\pi} \cdot v =x ) = O \left( \frac{1} { \epsilon \ n^{3/2} } \right).$$
\end{corollary}

Through slight additional work (see proof in Section 3), we also obtain the following
\begin{corollary}\label{thm:main_anti_conc}  
                Let $w, v \in \R^n$ with $\Delta(v) > 0$. Let $\overline w = n\inv\sum w_i$ and $\sigma^2 = \sum (w_i-\overline w)^2$. Then for any interval $I$,
                \begin{equation*}
                \pr  ( w_\pi \cdot v \in I) = O\left( \frac{ |I|}{n\sigma\Delta(v)}\sqrt{\log n}+ \frac 1n\right).
                \end{equation*}
\end{corollary}

The main tools in the proof of Theorem \ref{thm:anti_conc_perm} are the following, which are of independent interest

\begin{lemma}    [Random subset sum with replacement] \label{randomset1}
                Let $A:= \{ a_1, \dots, a_n \}$ be  a set of distinct real numbers. Set $S$ be the sum of $k$ elements from $A$, chosen uniformly randomly with replacement.  Then for any $x \in \R$,
                \begin{equation*}
                    \pr (S =x ) = O\left( \frac {1}{ n \sqrt k }\right) .
                \end{equation*}
\end{lemma}

\begin{lemma}   [Random subset sum without replacement]

\label{randomset2} Let $\{ a_1, \dots, a_n \}$ be  a set of distinct real numbers. Set $S$ be the sum of $k$ elements from $A$, chosen uniformly randomly without replacement.  Then for any $x \in \R$,
\begin{equation*}
\pr (X =x ) = O\left( \frac {1}{ n \sqrt {\min(k,n-k)} }\right) .
\end{equation*}
\end{lemma}

\noindent \textbf{Remark:}
We originally proved the main result of this note  around 2019--2020, as a tool to study random polynomials. The application we had in mind, however, has turned out to be more technical than anticipated, so we did not publish our findings at the time. Since then, we found out about recent developments \cite{HMS, Paw} and have decided to publish this first as it is of independent interest.

\section {Proofs of the main lemmas}
We make use of the following  result of Rogozin \cite{Rog}. Let $Y$ be a real random variable; recall the notion of concentration function

$$Q(Y, t) := \max _y \pr (y \le Y \le y +t)  . $$ 

\begin{theorem} [Rogozin]  \label{Rog}  There is a constant $C_0 >0$ such that the following holds. Let $X_1, \dots, X_k$ be independent random variables and set $X= \sum_I^k X_i$. Then
$$Q (X, t)  \le C_0 (\sum_{i=1} ^k (1- Q(X_i, t) )  ) ^{-1/2} . $$
\end{theorem}

With this in place, we are ready to prove the relevant lemmas.

\begin{proof}[Proof of Lemma \ref{randomset1}]
Without loss of generality, we can assume that $n \ge 10$.  We use an {\it induction} on $k$.  Set $C= 6(C_0+1) $.   Since $\pr (X=x)  \le n^{-1} $ for any $x$, the claim is trivial for $k \le 4$. From now we can assume $k >4$.

Now let $I$ be the shortest closed interval such that $| I \cap v | = \lceil n/2 \rceil$ and set $t = |I|$ (if there are many such intervals, take any).  By the definition of $I$ and $t$,  $Q(X_j, t) \le \lceil n/2 \rceil / n < 2/3 $, for all $1 \le j \le k$. Notice that

$$\pr ( X =x) =  \sum_{i=1}^n  \pr (X'= x -a_i  \wedge X_k = a_i )=   n^{-1} \sum_{i=1}^n   \pr (X'= x -a_i )  $$

\noindent Let $J:= x- I := \{ x-y,  y \in I\}.$  Then

$$ \sum_{i=1}^k \pr (X'= x -a_i )  = \pr (X' \in J) + \sum_{j, a_j \notin I} \pr (X' = x -a_j ) . $$

By Theorem \ref{Rog},   $\pr (X' \in J)  \le C_0 \sqrt {3/(k-1) }$. Furthermore, by the induction hypothesis  $$\pr (X' = x -a_j )  \le C (n \sqrt {k-1} )^{-1} . $$

\noindent  Thus,

$$ n^{-1}  \sum_{j, a_j \notin I} \pr (X' = x -a_j )  \le  \frac{C}{2}  (n \sqrt {k-1}  )^{-1}, $$ 
which---thanks to the fact that $k \ge 5$ and the definition of $C$---implies that
$$\pr (X =x) \le \frac{1}{ n \sqrt {k-1}} ( C_0 \sqrt  3  +  C/2 )\le C  \frac{1}{ n \sqrt {k}}, $$ 
as desired.
\end{proof}

\begin{proof}[Proof of Lemma \ref{randomset2}]We consider two cases.
\vskip2mm

\textbf{Case 1 ($k \leq n^{0.49}$):}  Suppose $k \le n^{.49}.$  Let $X_1, \dots, X_k$  be chosen uniformly from $A$ (with replacement).  Let $\mathcal{E} $ be the event that the values of $X_1, \dots, X_k$ are different. As $k= o(\sqrt n) $, by the birthday paradox, $\pr (\mathcal {E} ) = 1-o(1) $. On the other hand, if $\mathcal {E} $ occurs then the set formed by $X_1, \dots, X_k$ is uniformly random among the set of all subsets of size $k$. Therefore

$$\pr  (S=x)= \pr (\sum_{i=1}^k X_i =x | \mathcal {E} )  = \pr (\sum_{i=1}^k X_i =x \wedge \mathcal {E} )  \pr (\mathcal {E})^{-1} \le  (1+o(1) ) \pr (\sum_{i=1}^k =x), $$ and the claim follows from Lemma \ref{randomset1}.

\textbf{Case 2 ($k > n^{0.49}$):} We now turn our attention to the case that $k > n^{.49} $.  Without less of generality, we can assume $k \le n/2$.  We use the following  result from Inverse Littlewood-Offord theory (see \cite{NVsurvey} for a general discussion and clarifying details).

\begin{lemma}\label{Inverse} (Sparse Inverse Erd\H{o}s-Littlewood-Offord) \cite{NVsurvey}  Let $c < 1$ and $C$  be positive constants. There are constants $R,A$ depending on $c, C$ such that the following holds.

Let $x_i$ be iid $(0,1)$ random variables with $ n^{-c} \le  \pr  (x_i=1 ) :=\alpha  \le 1/2$.  Let $b_1, \dots, b_n$ be integers such that

$$ \rho:= \sup_a  \pr (b_ 1x_1 +\cdots+ b_n x_n = a) \ge n^{-C} . $$

Then  there exists a proper symmetric generalized arithmetic progression $Q$ of rank $r \le R $  which contains  at least $0.9n$ of the $v_i$ (counting multiplicities) and where $|Q|  \le A  \rho^{-1} (\alpha n)^{-r/2}  +1.$
\end{lemma}

In order to use the above result, we set $b_i =a_i + N$ where $N$ is a large integer to be chosen later. Towards a contradiction, assume that for some $x$

\begin{equation} \label{assumption1} \pr (S =x )  \ge  B  \frac {1}{ n \sqrt k } , \end{equation}, where  $B >0$ is a large constant to be chosen.

Define $x_i$ as in Lemma \ref{Inverse} with $\pr (x_i =1) = \alpha k/n $.  The probability that exactly $k$ of the $x_i$ are non-zero is at least $c_1 k^{-1/2} $ for some constant $c_1 >0$. Thus, \eqref{assumption1} implies that

$$ \rho \ge  \pr (b_1 x_1 + \dots +b_n x_n = x + Nk ) \ge B \frac{1}{ n \sqrt k} c_1 k^{-1/2} = B c_1 (nk)^{-1} . $$

Now we apply Lemma \ref{Inverse}. By choosing $N$ sufficiently large (say $N \ge 2^n$), it is easy to see that there is no symmetric arithmetic progression of length $O(n^2) $ containing at least half of the $b_i$. Thus, the rank $r$ in this lemma is at least $2$. By the conclusion of the lemma $$|Q| \le A B^{-1} c_1^{-1} nk (\alpha n)^{-1}  +1  . $$

By setting $B \ge 2 A c_1^{-1} $, and recalling that $\alpha = k/n$, we have

$$|Q| \le \frac{Ac_1} {B} n +1 \le \frac{n}{2} +1 < 0.9 n , $$  a contradiction.
\end{proof}

\section{Proof of Theorem  \ref{thm:anti_conc_perm} and Corollary \ref{thm:main_anti_conc}}

\begin{proof}[Proof of Theorem \ref{thm:anti_conc_perm}]
Set $\beta (n,k) := \frac {1}{ n \sqrt {\min(k,n-k)} }$.  Recall that  $w = (w_1, \ldots, w_n)$ and $v = (v_1, v_2, \ldots , v_n)$  are increasing and $\Delta = \min_{i\neq j} |v_i - v_j| \neq 0$.

For any $t$, let  $A(t) = \{ i : w_i < t\}$ , $B(t) = \{ i : w_i \geq t\}$. For a permutation $\sigma$, let  $S (\sigma)  = \sum_{i=1} ^{n} w_{\sigma(i)} v_i$.  Fix two numbers $t_1 <t_2$ and let   $A := A(t_1)$ and $B := B(t_2)$.  Then we  define a permutaiton $\sigma$ via the following process:

\begin{itemize}

                \item[(i)] pick $\sigma(i)$ for each $i \notin A \cup B$;

                \item[(ii)] determine the relative ordering of $\sigma(i)$ and $\sigma(j)$ for each $\{i,j\} \subseteq A$ and each $\{i,j\} \subseteq B$;

                \item[(iii)] determine the set $\{ \sigma(i) \ : \ i \in A\}$.

\end{itemize}

Notice that if  $\sigma$ is chosen uniformly at random from $S_n$, then given any outcomes for (i) and (ii), the set in (iii) is distributed uniformly from among all $|A|$-element subsets of $\sigma(A) \cup \sigma(B)$.

Suppose the outcomes of (i) and (ii) are given.  Let $X = \sigma(A) \cup \sigma(B)$, and for each $U \subseteq X$ with $|U| = |A|$, let $S(U)$ denote the sum $\sum_{i=1} ^{n} w_{i} v_{\sigma(i)}$ where $\sigma$ is the permutation that would result from chosing $\sigma(A) = U$ in step (iii).

Define the following partial order $P$ on $|A|$-element subsets of $X$: $U \preceq V$ iff for all $x$ we have $|U \cap (-\infty, x]| \leq |V \cap (-\infty, x]|$.  Stanley \cite{Stanley} proved  that  this poset has Sperner property, which implies  that its width  is given by

                \[
                width(P) = \max_{t} \left| \left \{ U \subseteq \{1, 2, \ldots, n\} \ : \ |U| = k, \text{ and } \sum_{u \in U} u = t \right \} \right|.
                \]

    Furthermore, Lemma \ref{randomset2} implies that the width of $P$ is $O({n \choose k} \beta (k,n)) $.

Next, notice that  $U \preceq V$ iff for all $i$, the $i^{th}$ largest element of $U$ is at least that of $V$.  Thus, if $U \prec V$ and there are no sets $W$ strictly between $U$ and $V$, then $U$ must be of the form $U = \{y\} \cup V \setminus \{x\}$ for some $y > x$.  In this case, we have
\[
S(V) - S(U) = (v_{x} - v_y) (w_{i} - w_{j}),
\]

for some elements $i \in A$ and $j \in B$.  It follows that  if $U \prec V$, then we have
\begin{equation} \label{keybound} S(V) - S(U) \geq \Delta (t_2 - t_1) > 0. \end{equation}

Let $E$ be the collection of sets $U$ for which $S(U) \in I$.  By Dilworth's theorem, we can decompose $E$ into chains $C_1, C_2, \ldots, C_l$, where  $l$ is the width of $E$.  By \eqref{keybound}, we know that each chain in $E$ has size at most $1 + \frac{|I|}{\Delta (t_2 - t_1)}$.  On the other hand, as argued above, the width is at most $O{|X| \choose |A|} \beta (|A|, |A| +|B|) $.

Therefore,
\[
|E| = O( \left(1 + \dfrac{|I|}{\Delta (t_2 - t_1)} \right) {|X| \choose |A|}\beta (|A|, |A|+|B|) )
\]
which implies
\[
\mathbb{P}(S(U) \in I) = O(  \left(1 + \dfrac{|I|}{\Delta (t_2 - t_1)} \right) \beta (|A|, |A|+ |B|)).
\]
This proves the theorem, via a routine conditioning argument.
\end{proof}

\begin{proof}[Proof of Corollary \ref{thm:main_anti_conc}]

Without loss of generality we may assume that $\overline w = 0$, so $\sigma^2 = \sum w_i^2$. Let $j$ be the largest index among negative entries of $w$, and assume without loss of generality that $j \le n/2$. Our first guess might be to choose $i_1 = k < j+1 = i_2$. If there is some index $k \le j$ such that $w_k^2 > \frac{\sigma^2}{100 k \log n}$ then by \cref{thm:anti_conc_perm} this choice yields the claimed bound. Otherwise, we have

\begin{equation}\label{eqn:less_than_zero}
    \sum_{k \le j} w_k^2 \le \sum_{k \le j} \frac{\sigma^2}{100 k \log n} \le \frac{\sigma^2}{100}.
\end{equation}

By Cauchy-Schwarz we conclude $\left|\sum_{k \le j} w_k\right| \le \sigma\sqrt{j}/10 \le \sigma \sqrt{n}/10$. On the other hand, since $\overline w = 0$ we have

\begin{equation*}
    \frac{\sigma\sqrt{n}}{10} \ge \left|\sum_{k \le j} w_k\right| = \left|\sum_{ k > j} w_k\right| \ge \frac{n}{2}w_{n/2}.
\end{equation*}

So we see $w_{n/2} \le \sigma/(5\sqrt{n})$. Since $w_{n/2}$ is so small, our next guess will be to take $i_1 = w_{n/2}$, $i_2 = w_{n-k}$ for some $k < n/2$. If $(w_k-w_{n/2})^2 > \frac{\sigma^2}{10 k \log n}$ then by \cref{thm:anti_conc_perm} this choice yields the corollary. Otherwise, $w_{n-k} \le \frac{\sigma}{5\sqrt{n}}+\frac{\sigma}{\sqrt{10k\log n}}$ and

\begin{equation}\label{eqn:zero_to_n/2}
    \sum_{k \le n/2} w_{n-k}^2 \le \sum_{k \le n/2} \frac{\sigma^2}{10n}+\frac{\sigma^2}{5k\log n} \le \frac{3\sigma^2}{10}.
\end{equation}

Finally, observe

\begin{equation}\label{eqn:n/2_and_up}
    \sum_{k \in [j, n/2]} w_k^2 \le \frac 12\sum_{k \ge j} w_k^2 \le \frac{\sigma^2}{2}.
\end{equation}

Combining (\ref{eqn:less_than_zero}), (\ref{eqn:zero_to_n/2}), and (\ref{eqn:n/2_and_up}) we see

\begin{equation*}
    \sigma^2 = \sum_{k} w_k^2 \le \frac{\sigma^2}{100}+\frac{\sigma^2}{2}+\frac{3\sigma^2}{10} < \sigma^2,
\end{equation*}
which is a contradiction, completing the proof.

\end{proof}

\end{document}